\documentclass[12pt,a4paper]{article}
\usepackage{amssymb,amsmath,amsthm}
\usepackage[english]{babel}
\usepackage{t1enc}
\usepackage[latin2]{inputenc}
\usepackage{epsfig}
\usepackage{subfigure}
\usepackage{footnpag}

\newtheorem{thm}{Theorem}
\newtheorem{cor}[thm]{Corollary}

\newtheorem{lem}[thm]{Lemma}

\newtheorem{claim}[thm]{Claim}

\usepackage{indentfirst}
\def\R{\mathbb{R}}

\begin{document}

\title{Octants are Cover Decomposable}
\author{Bal\'azs Keszegh and D\"om\"ot\"or P\'alv\"olgyi
}
\maketitle

\begin{abstract}
We prove that octants are cover-decomposable, i.e., any $12$-fold covering of any subset of the space with a finite number of translates of a given octant can be decomposed into two coverings. As a corollary, we obtain that any $12$-fold covering of any subset of the plane with a finite number of homothetic copies of a given triangle can be decomposed into two coverings. We also show that any $12$-fold covering of the whole plane with open triangles can be decomposed into two coverings. However, we exhibit an indecomposable $3$-fold covering.
\end{abstract}

\medskip

\section{Introduction}
Let ${\cal P}=\{\ P_i\ |\ i\in I\ \}$ be a collection of geometric sets in $\R^d$.
We say that ${\cal P}$ is an {\em $m$-fold covering} of a set $S$ if every point of $S$ is contained in at least $m$ members of $\cal P$. A $1$-fold covering is simply called a {\em covering}. 

\medskip

\noindent {\bf Definition.} A geometric set $P\subset \R^d$ is said to be {\em cover-decomposable}\footnote{In the literature the definition is slightly different and the notion defined here is sometimes called {\em finite-cover-decomposable}, however, to avoid unnecessary complications, we simply use {\em cover-decomposable}.} if there exists a (minimal) constant $m=m(P)$ such that
every $m$-fold covering of any subset of $\R^d$ with a finite number of translates of $P$ can be decomposed into two coverings of the same subset. Define $m$ as the cover-decomposability constant of $P$.

\medskip

The simplest objects to examine are the orthants 
 of $\R^d$. It is easy to see that a quadrant ($2$-dimensional orthant) is cover-decomposable. Cardinal \cite{C11} noticed that orthants in $4$ and higher dimensions are not cover-decomposable as there is a plane on which their trace can be any family of axis-arallel rectangles and it was shown by Pach, Tardos and T\'oth \cite{PTT09} that such families might not be decomposable into two coverings. Cardinal asked whether octants ($3$-dimensional orthants) are cover-decomposable. Our main result is an affirmative answer (the proof is in Section \ref{secmain}). 

\begin{thm}\label{mainthm} Octants are cover-decomposable, i.e., any $12$-fold covering of any subset of $\R^3$ with a finite number of translates of a given octant can be decomposed into two coverings.
\end{thm}

The intersection of the translates of the octant containing $(-\infty,-\infty,-\infty)$ with the $x+y+z=0$ plane gives the homothetic copies of an equilateral triangle. Since any triangle can be obtained by an affine transformation of the equilateral triangle we obtain 

\begin{cor}\label{cortriangle} Any $12$-fold covering of any subset of the plane with a finite number of homothetic copies of any given triangle can be decomposed into two coverings.
\end{cor}

Using standard compactness arguments, this implies the following (the proof is in Section \ref{locsct}):

\begin{thm}\label{locthm} Any locally finite\footnote{We say that a covering is {\em locally finite} if every compact set intersects only a finite number of covering sets, i.e. homothetic copies of the given triangle, in our case.}, $12$-fold covering of the whole plane with homothetic copies of a triangle is decomposable into two coverings.
\end{thm}

The analogs of Corollary \ref{cortriangle} and Theorem \ref{locthm} for translates of a given triangle were proved with a bigger constant by Tardos and T\'oth \cite{TT07} using a more complicated argument\footnote{The original proof gave $m=43$ which was later improved by \'Acs \cite{A10} to $m=19$ which is still worse than our $12$.}.
Following their idea, using Theorem \ref{locthm} for translates of a given open triangle we obtain

\begin{cor}\label{corplane}
Any $12$-fold covering of the whole plane with the translates of an open triangle is decomposable into two coverings. 
\end{cor}

Our result brings the task to determine the exact cover-decomposability constant of triangles in range. Tardos and T\'oth state that they cannot even rule out the possibility that the cover-decomposability constant of triangles is $3$. To complement our upper bound, in Section \ref{seclower} we show a construction proving that the constant is actually at least $4$.

Our proof of Theorem \ref{mainthm} in fact proves the following equivalent, dual\footnote{For more on dualization and for the proof of equivalence, see the surveys \cite{P10} and \cite{PPT11}.} form of Theorem \ref{mainthm}.

\begin{thm}\label{dualthm} Any finite set of points in $\R^3$ can be colored with two colors such that any translate of a given octant with at least $12$ points contains both colors.
\end{thm}

Finally, we mention the dual of Corollary \ref{cortriangle}, which is not equivalent to Corollary \ref{cortriangle} but follows from Theorem \ref{dualthm} the same way as Corollary \ref{cortriangle} follows from Theorem \ref {mainthm}.

\begin{cor} Any finite planar point set can be colored with two colors such that any homothetic copy of a given triangle that contains at least $12$ points contains both colors.
\end{cor}

For more results on cover-decomposability, see the recent surveys \cite{P10} and \cite{PPT11}.

\section{Proof of Theorem \ref{mainthm} and \ref{dualthm}} \label{secmain}
Denote by $W$ the octant with apex at the origin containing $(-\infty,-\infty,-\infty)$.
We will work in the dual setting, that is we have a finite set of points, $P$,
in the space, that we want to color with two colors such that any translate of $W$ with at least $12$ points contains both colors. If we can do this for any $P$, then it follows using a standard dualization argument (see \cite{P10} or \cite{PPT11}) that $W$ (and thus any octant) is cover-decomposable. So from now on our goal will be to show the existence of such a coloring.

For simplicity, suppose that no number occurs multiple times among the coordinates of the points of $P$ (otherwise, by a small perturbation of $P$ we can get such a point set, and its coloring will be also good for $P$). Denote the point of $P$ with the $t^{th}$ smallest $z$ coordinate by $p_t$ and the union of $p_1,\ldots,p_t$ by $P_t$. First we will show how to reduce the coloring of $P_t$ to a planar and thus more tracktable problem.

Denote the projection of $P$ on the $z=0$ plane 
by $P'$. Similarly denote the projection of $p_t$ by $p_t'$, the projection of $P_t$ by $P_t'$ and the projection of $W$ by $W'$. Therefore $W'$ is the quadrant with apex at the origin containing $(-\infty,-\infty)$.

\begin{claim}\label{project} We can color the points of the planar point set, $P'$, with two colors such that for any $t$ and any translate of $W'$ containing at least $12$ points of $P_t'$, it is true that the intersection of this translate and $P_t'$ contains both colors if and only if we can color the points of the spatial point set, $P$, with two colors such that for any translate of $W$ containing at least $12$ points of $P$, it is true that the intersection of this translate and $P$ contains both colors.
\end{claim}
\begin{proof}
Clearly, if we take a translate of $W$ with apex $w$ having $z$ coordinate bigger than the $z$ coordinate of $p_t$ and smaller than the $z$ coordinate of $p_{t+1}$, then the projection of the intersection of this translate with $P$ is equal to the intersection of $P_t'$ with the translate of $W'$ having apex $w'$, the projection of $w$.
Thus having a good coloring for one problem gives a good coloring for the other if we give $p_t$ and $p_t'$ the same color for every $t$.
\end{proof}

In the rest of this paper, such a coloring of a planar point set is called a {\em good coloring}. 
Now we will prove that any $P'$ has a good coloring, thus establishing Theorem \ref{dualthm} and since they are equivalent, also Theorem \ref{mainthm}.
To avoid going mad, we will omit the apostrophe in the following, so we will simply write $W$ instead of $W'$ and so on.
Also, we will use the term {\em wedge} to denote a translate of $W$. 

A possible way to imagine this planar problem is that in every step $t$ we have a set of points, $P_t$, and our goal is to color the coming new point, $p_{t+1}$, such that we always have a good coloring.
We note that this would be impossible in an online setting, i.e. without knowing in advance which points will come in which order. But using that we know in advance every $p_i$ makes the problem solvable.

We start by introducing some notation.
If $p_x<q_x$ but $p_y>q_y$ then we
say that $p$ is NW from $q$ and $q$ is SE from $p$. In this case we call $p$ and $q$
incomparable. Similarly, $p$ is SW from $q$ (and $q$ is NE from $p$) if and only if 
both coordinates of $p$ are smaller than the respective coordinates of $q$.

Instead of coloring the points, we will rather define on them a bipartite graph $G$, whose proper two-coloring will give us a good coloring. Actually, as we will later see, this graph will be a forest. 

We define $G$ recursively, starting with the empty set and the empty graph. At any step $j$ we define a graph $G_{j}$ on the points of $P_{j}$ and also maintain a set $S_{j}$ of pairwise incomparable points, called the {\em staircase}. Thus, before the $t^{th}$ step we have a graph $G_{t-1}$ on the points of $P_{t-1}$ and a set $S_{t-1}$ of pairwise incomparable points. In the $t^{th}$ step we add $p_t$ to our point set obtaining $P_t$ and we will define the new staircase, $S_t$, and also the new graph, $G_t$, which will have $G_{t-1}$ as a subgraph. Before the exact definition of $S_t$ and $G_t$, we make some more definitions and fix some properties that will be maintained during the process.

In any step $j$, we say that a point $p$ is {\em good} if any wedge containing $p$ already contains two points of $P_j$ connected by an edge of $G_j$. I.e. at any time after $j$ a wedge containing $p$ will contain points of both colors in the final coloring.

At any time $j$, consider the order of the points of $S_j$ given by their $x$ coordinates. If two points of $S_j$ are consecutive in this order then we say that these staircase points are {\em neighbors}\footnote{This does not mean that they are connected in the graph.}. 
A point $s$ of the staircase is {\em almost good} if any wedge containing $s$ and its neighbor(s) contains two points of $P_j$ connected by an edge of $G_j$. Notice that the good points and the neighbors of the good points are always almost good.

We say that a point $p$ of $P_j$ is {\em above} the staircase if there exists a staircase point $s\in S_j$ such that $p$ is NE from $s$.
If $p$ is not above or on the staircase, then we say that $p$ is {\em below} the staircase.
Now we can state the properties we maintain:

At any time $j$:
\begin{enumerate}
\item All points above the staircase are good.
\item All points of the staircase are almost good.
\item All points below the staircase are incomparable.
\item If a wedge only contains points that are below the staircase then it contains at most $3$ points.
\end{enumerate}

For $t=0$ all these properties are trivially true. Suppose that the properties hold at time $t-1$. Now we proceed with point $p_t$ according to the following algorithm maintaining all the properties. During the process, we denote the actual graph by $G$ and the actual staircase by $S$.\\

{\bf \textsc Algorithm Step $t$}
 
Set $G=G_{t-1}$ and $S=S_{t-1}$.

\begin{figure}[t]
    \centering
        \includegraphics[scale=1]{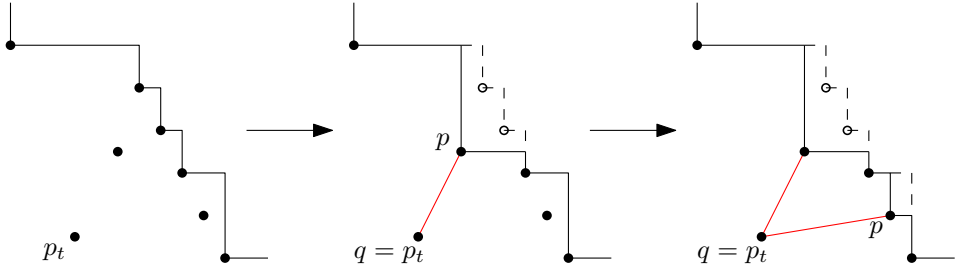}
   \caption{Repeated application of Step (b) of the algorithm}
   \label{case2}
\end{figure}

\begin{itemize}
\item[Step (a)] If $p_t$ is above the staircase $S_{t-1}$ then we do the following, otherwise skip to Step (b).

In this case $S_t=S_{t-1}$ and $G_t=G_{t-1}\cup \{e\}$, where $e$ is an arbitrary edge between $p_t$ and a point $s$ of $S_{t-1}$ which is SW from $p_t$. The properties will hold trivially by induction, the only thing we need to check is if $p_t$ is a good point, but this is again guaranteed as any wedge containing $p_t$ contains the edge $e$. The algorithm terminates. 

Note that we proceed further if and only if $p_t$ is below the staircase $S_{t-1}$.

\item[Step (b)] If there exist two points $p$ and $q$ that are below the staircase and $p$ and $q$ are comparable 
then we do the following, otherwise skip to Step (c).

Without loss of generality suppose that $q$ is SW from $p$.
Notice that because of Property 3, either $p$ or $q$ is the last added point and there are no points below the staircase that are NE from $p$. Now, define the new staircase, $S$, as $S$ minus the points of $S$ that are NE from $p$, plus the point $p$. This way the points of the staircase remain pairwise incomparable as we added $p$ and deleted all the points that were comparable to $p$. Also, we add the edge $pq$ to the graph, i.e. $G:=G\cup \{pq\}$. For an illustration of repeated application of this step, see Figure \ref{case2} (edges of G are drawn red). Thus, any wedge containing $p$ contains the edge $pq$, i.e. $p$ is a good point. Property 1 is true for the points that were above the old $S$ by induction. All other points above $S$ are exactly the points that were deleted from the staircase in this step. All such points are NE from $p$ and thus any wedge containing them contains the edge $pq$. Property 2 holds for $p$ as it is a good point, it holds for the 2 neighbors of $p$ as any point neighboring a good point is an almost good point. For any other $s$ from the staircase its neighbors remain the same, so it remains almost good.

Go back to Step (b) until Property 3 is satisfied, then proceed to Step (c).

\item[Step (c)] If there exist $4$ points below the staircase such that these $4$ points are pairwise incomparable and there exists a wedge $V$ such that $V$ contains these $4$ points but no points of the staircase then do the following, otherwise skip to Step (d).

Denote these $4$ points by $q_1,q_2,q_3,q_4$ in increasing order of their $x$ coordinates. Notice that there are no points below the staircase that are comparable because of Step (b). Now define the new $S$ as the old $S$ minus the points of $S$ that are NE from $q_2$ or $q_3$, plus the points $q_2$ and $q_3$. This way the points of the staircase remain pairwise incomparable as we added $q_2$ and $q_3$ and deleted all the points that were comparable to them. Also, we add the edges $q_1q_2$ and $q_3q_4$ to the graph, i.e. $G_t=G_{t-1}\cup \{q_1q_2,q_3q_4\}$. For an illustration see Figure \ref{case3}. Property 1 is true for the points that were above the old $S$ by induction. All other points above the new $S$ are exactly the points that were deleted from the staircase in this step. It is easy to check that such a point is either NE from both of $q_1$ and $q_2$ or it is NE from both of $q_3$ and $q_4$ (we use that $V$ was completely below the staircase, see Figure \ref{case3b}). Thus, a wedge containing such a point contains the edge $q_1q_2$ or the edge $q_3q_4$, Property 1 will be true. Property 2 is true for $q_2$ and also for its neighbor which is not $q_3$ as a wedge covering them must cover $q_1$ as well and thus the edge $q_1q_2$, i.e. they are almost good. By symmetry $q_3$ and its neighbor which is not $q_2$ are also almost good. For any other $s$ from the rest of the staircase, $s$ remains almost good by induction as its neighbors do not change.

Go back to Step (c) until Property 4 is satisfied, then proceed to Step (d).

\begin{figure}[t]
    \centering
    \subfigure[]{\label{case3}
        \includegraphics[scale=1]{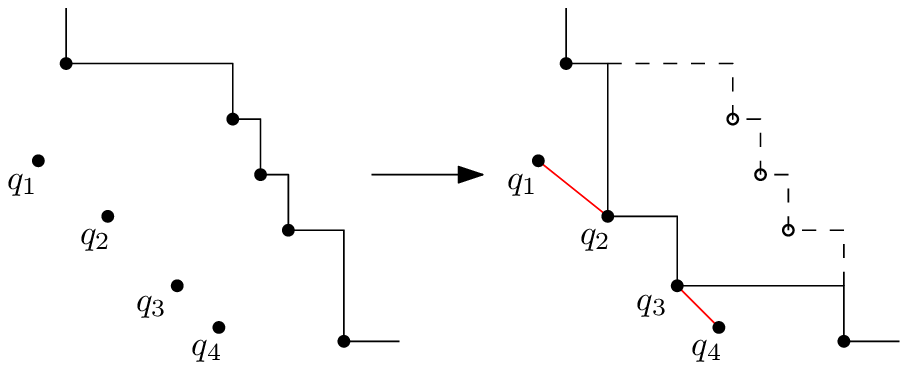}
        }
    \hskip 10mm
    \subfigure[]{\label{case3b}
        \includegraphics[scale=1]{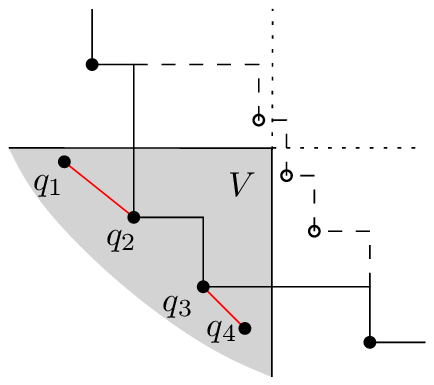}
        }
   \caption{Application of Step (c) of the algorithm}       
\end{figure}

\item[Step (d)] Set $S_t=S$ and $G_t=G$ and the algorithm terminates. 

\end{itemize}

Adding $p_t$ below the staircase and proceeding as in Case (b) or Case (c) always maintains Properties 1 and 2. As neither Case (b) nor Case (c) can be applied anymore, Properties 3 and 4 must hold as well. 
Now let us examine the graph $G$.

\begin{claim}
The final graph $G$ is a forest.
\end{claim}
\begin{proof}
We prove by induction a stronger statement, that $G$ will be such a forest that the components of the points below the staircase are disjoint trees.

When we add an edge in Step (a), then the newly added point that goes above the staircase will be one of the endpoints, thus this property is maintained.

When we add an edge in Step (b) or (c), then it connects two points below the staircase one of which we immediately move to the staircase, so we are done by induction.
\end{proof}

\begin{figure}[t]
    \centering
        \includegraphics[scale=.8]{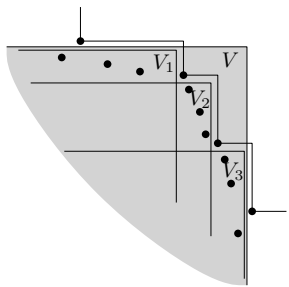}
   \caption{At most 11 points can be in a monochromatic wedge}
   \label{12point}
\end{figure}

\begin{claim}
Any two-coloring of $G$ is a good coloring of $P$.
\end{claim}
\begin{proof}
Take an arbitrary two-coloring of $G$. Take an arbitrary wedge $V$ at time $t$ that contains at least $12$ points of $P_t$. If $V$ contains a point from above the staircase $S_t$ then by Property 1 $V$ contains points of both colors. If $V$ contains at least $3$ points from the staircase then $V$ also contains $3$ consecutive points, thus by applying Property 2 to the middle one $V$ contains both colors (as it contains both neighbors of this middle point). Finally, if a wedge $V$ does not contain a point from above the staircase and contains at most $2$ points from the staircase then all the points below the staircase that are covered by $V$ can be covered by $3$ wedges containing points only from below the staircase (see the three wedges $V_1, V_2$ and $V_3$ in Figure \ref{12point}). By Property 4 each of these wedges cover at most $3$ points, thus $V$ can contain altogether at most 11 points ($2$ from the staircase and $3\cdot 3$ from below the staircase), a contradiction.
\end{proof}

The above claim finishes the proof of the theorem.

\section{Miscellany and a lower bound}\label{seclower}
We have seen in the Introduction that if the point set of Theorem \ref{dualthm} is from the $x+y+z=0$ plane, then the problem is equivalent to the cover-decomposability of homothetic copies of an equilateral triangle. Another special case is if the point set is on the $x+y=0$ plane. 
The intersection family of the octants with this plane is the family of {\em bottomless axis-parallel rectangles}\footnote{A set is a {\em bottomless axis-parallel rectangle} if it is the homothetic copy of the set $\{(x,y):0<x<1,y<0\}$ 
.}. 
Bottomless rectangles were examined by the first author \cite{wcf} where it was proved that any $3$-fold covering with bottomless rectangles is decomposable into two coverings and also that any finite point set can be colored with two colors such that every bottomless rectangle containing at least $4$ points contains both colors. It was also shown  that these results are sharp.
We will use the ideas from \cite{wcf} to prove the following claims, first of which is a strengthening of Theorem \ref{dualthm} in a special case and the second giving a sharp lower bound for this special case, which also holds for the general case. 

\begin{claim}\label{atlonclaim}
If the projection of the original point set from $R^3$ onto the $z=0$ plane yields a point set $P$ having only pairwise incomparable points, then it admits a two-coloring such that any translate of a given octant that contains at least $4$ points contains points with both colors.
\end{claim}
\begin{proof}
We use the same notations as in Section $2$. 
Now at any time the points of $P_t$ are pairwise incomparable. Order them according to their $x$ coordinate. We will maintain a partial coloring such that at any time $t$:
\begin{itemize}
\item[1.] There are no two consecutive points in this order that are not colored.
\item[2.] The colored points are colored alternatingly.
\end{itemize}
We start with the empty set and then in a general step we add the point $p_t$. If in the order it goes between two colored points then we leave it uncolored. If it comes next to an uncolored point then we color these two points maintaining the alternating coloring. At the end we color the remaining uncolored points arbitrarily. We claim that at any time $t$ any wedge covering at least $4$ points covers points from both colors already at step $t$ of the coloring. Indeed, any wedge covers consecutive points and it covers at least $2$ (consecutive) colored points by Property 1 and any two consecutive colored points are colored differently by Property 2.
\end{proof}
\begin{figure}[t]
    \centering
    \subfigure[]{\label{constrdiag}
        \includegraphics[scale=0.7]{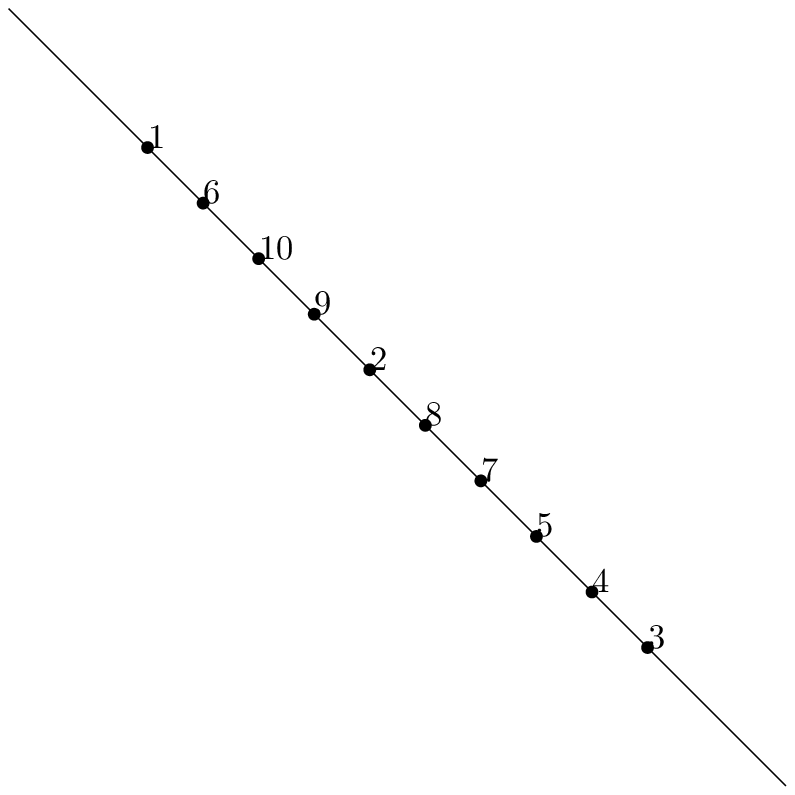}}
    \subfigure[]{\label{constr}
        \includegraphics[scale=0.7]{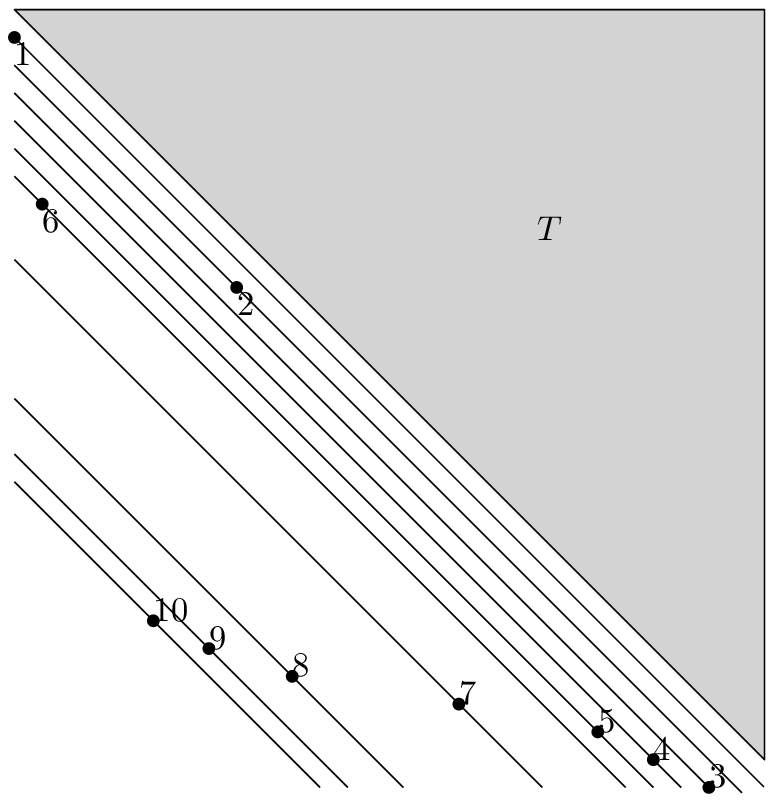}
}
   \caption{Lower bound constructions}
   \label{}
\end{figure}

\begin{claim}
For any octant there exists a $10$ point set $P\subset R^3$ such that its projection onto the $z=0$ plane yields a point set having only pairwise incomparable points, yet in any two-coloring of $P$ there exists a translate of a given octant that contains $3$ points with the same color and no other points.
\end{claim}
\begin{proof}
The point set on Figure \ref{constrdiag} has the needed properties (for simplicity, the projection of the point having the $t^{th}$ biggest $z$ coordinate is denoted by $t$ instead of $p_t$). Indeed, suppose on the contrary that there exists a two-coloring with no monochromatic wedge covering exactly $3$ points. It is easy to check that the triples $(1,2,3),$ $(1,2,4),$ $(1,2,5),$ $(3,4,5),$ $(6,2,5),$ $(6,2,7),$ $(6,2,8),$ $(5,7,8),$ $(6,1,2),$ $(6,1,9),$ $(6,1,10),$ $(2,9,10)$ can all be covered by some wedge at some time $t$. By the pigeonhole principle there are two points from $(1,2,6)$ that have the same colors. If e.g. $1$ and $2$ are colored red, then by the first three triples $3, 4$ and $5$ all must be colored blue, but then the fourth triple is monochromatic, a contradiction. The analysis is similar if $2$ and $6$ have the same color. Finally, if $1$ and $6$ is e.g. red and $2$ is blue, then we obtain a contradiction from the last three triples, as $9$ and $10$ should be both blue because of the penultimate and antepenultimate triples, but then the ultimate triple is monochromatic.
\end{proof}

This construction can be modified a bit to imply the same result for translates of a given triangle.

\begin{claim}
There exists a 10 point set $P\subset R^2$ and a given triangle $T$ such that in any two-coloring of $P$ there exists a translate of $T$ that contains $3$ points of the same color and no other points.
\end{claim}
\begin{proof}
The point set and the triangle on Figure \ref{constr} has the needed properties, the proof of this is exactly the same as of the previous claim.
\end{proof}

Finally we note that this construction is a bit smaller then the one in \cite{wcf}, which had size $12$, so we obtain a smaller construction for that problem too by taking the same points as in Claim \ref{atlonclaim} projected onto the $y=0$ plane.

\section{Locally finite coverings of the whole plane} \label{locsct}
In this section we prove Theorem \ref{locthm}.

We say that a covering is {\em locally finite} if every compact set intersects only a finite number of covering sets, i.e. homothetic copies of the given triangle, in our case.
In this section we prove that any locally finite, $12$-fold covering of the whole plane with homothetic copies of a triangle is decomposable into two coverings. After an affine transformation we can suppose that the triangle is equilateral, we will denote it by $T$.
We will use

\begin{lem}\label{kil}[K\"onig's Infinity Lemma, \cite{KIL}] Let $V_0,V_1,..$ be an infinite sequence of disjoint non-empty finite sets, and let $G$ be a graph on their union. Assume that every vertex $v_n$ in a set $V_n$ with $n\ge1$ has a neighbour $f(v_n)$ in $V_{n-1}$. Then $G$ contains an infinite path, $v_0v_1...$ with $v_n \in V_n$ for all $n$.
\end{lem}

Take $K_1\subset K_2\subset \ldots$ compact sets such that their union is the whole plane. Let each $v_n$ be a possible coloring of those finitely many triangles that intersect $K_n$ such that every point of $K_n$ is covered by both colors. In this case $V_n$ is non-empty because of Corollary \ref{cortriangle}. The function $f$ is the natural restriction to the triangles that intersect $K_{n-1}$. The infinite path gives a partition to two coverings.\hfill$\Box$

\subsection*{Remarks and acknowledgment}
We would like to thank Jean Cardinal for calling our attention to this problem.

\end{document}